\documentclass[pdftex,a4paper,12pt]{scrartcl}

\usepackage[T1]{fontenc}
\usepackage{lmodern} 
\usepackage{exscale}

\usepackage[pdftex]{graphicx}
\usepackage{amsmath,amssymb,amsthm}
\usepackage{mathtools}
\usepackage{hyperref}
\usepackage[capitalize]{cleveref}
\usepackage{enumitem}

\usepackage{macros}
\usepackage[tikzset={x=.6cm, y=.6cm}]{linkdiags} 

\allowdisplaybreaks[2]

\ifdefined\cprime\else
\def\cprime{'}
\fi

\defTikzBox[baseline=-.5ex,scale=.7]{arrayHCrossNeg}{%
  \draw[red,dashed,very thick] (-1,1) arc (90:270:1);
  \draw[red,dashed,very thick] (1,-1) arc (-90:90:1);
  \draw[red,very thick] (-1,-1) .. controls +(.5,0) and +(-.5,0) .. (1,1);
  \draw[white,line width=5] (-1,1) .. controls +(.5,0) and +(-.5,0) .. (1,-1);
  \draw[red,very thick] (-1,1) .. controls +(.5,0) and +(-.5,0) .. (1,-1); }

\defTikzBox[baseline=-.5ex,scale=.7]{arrayHCrossPos}{%
  \draw[red,dashed,very thick] (-1,1) arc (90:270:1);
  \draw[red,dashed,very thick] (1,-1) arc (-90:90:1);
  \draw[red,very thick] (-1,1) .. controls +(.5,0) and +(-.5,0) .. (1,-1);
  \draw[white,line width=5] (-1,-1) .. controls +(.5,0) and +(-.5,0) .. (1,1); 
  \draw[red,very thick] (-1,-1) .. controls +(.5,0) and +(-.5,0) .. (1,1); }

\defTikzBox[baseline=-.5ex,scale=.7]{arrayVCrossNeg}{%
  \draw[red,dashed,very thick] (1,1) arc (0:180:1);
  \draw[red,dashed,very thick] (-1,-1) arc (180:360:1);
  \draw[red,very thick] (-1,1) .. controls +(0,-.5) and +(0,.5) .. (1,-1);
  \draw[white,line width=5] (1,1) .. controls +(0,-.5) and +(0,.5) .. (-1,-1);
  \draw[red,very thick] (1,1) .. controls +(0,-.5) and +(0,.5) .. (-1,-1); }

\defTikzBox[baseline=-.5ex,scale=.7]{arrayVCrossPos}{%
  \draw[red,dashed,very thick] (1,1) arc (0:180:1);
  \draw[red,dashed,very thick] (-1,-1) arc (180:360:1);
  \draw[red,very thick] (1,1) .. controls +(0,-.5) and +(0,.5) .. (-1,-1);
  \draw[white,line width=5] (-1,1) .. controls +(0,-.5) and +(0,.5) .. (1,-1);
  \draw[red,very thick] (-1,1) .. controls +(0,-.5) and +(0,.5) .. (1,-1); }

\theoremstyle{plain}
\newtheorem{mainthm}{Main Theorem}
\newtheorem*{maincor}{Corollary}
\newtheorem{theorem}{Theorem}[section]
\newtheorem{proposition}[theorem]{Proposition}
\newtheorem{corollary}[theorem]{Corollary}
\newtheorem{lemma}[theorem]{Lemma}

\theoremstyle{definition}
\newtheorem{definition}[theorem]{Definition}

\theoremstyle{remark}
\newtheorem{example}[theorem]{Example}
\newtheorem{remark}[theorem]{Remark}
\newtheorem{notation}[theorem]{Notation}

\numberwithin{equation}{section}
\numberwithin{figure}{section}

\title{Crossing change on Khovanov homology and a categorified Vassiliev skein relation}
\date{\today}
\author{Noboru Ito and Jun Yoshida}

\begin{document}
\maketitle

\tableofcontents

\section{Overview}
\emph{Crossing change} is the basic notion in knot theory.
Historically, the unknotting number is one of the classical measures of the complexity of knots, and 
skein relations define  
 knot polynomials: e.g. the Jones, HOMFLY-PT, Alexander-Conway polynomials.
In this paper, we categorify the notion of \emph{crossing change} on Khovanov homology.

In 1990, Vassiliev \cite{Vassiliev1990} studied the stratified space of knots.  The computation of  the associated spectoral sequence gives rise to a sort of classes of knot invariants, which are called Vassiliev invariants.   Vassiliev invariants are extended to link or braid invariants.  
Many works have been carried out on Vassiliev invariants; in particular, it is known that Vassiliev  invariants are at least as powerful as all of the quantum group invariants of knots (Birman \cite{Birman1993}) and that they classify pure braids (Kohno \cite{Kohno1993}).
In Vassiliev theory, crossing change has a crucial role in defining the \emph{order} that yields a filtration on the space of \emph{Vassiliev invariants} (aka. \emph{finite type invariants}).  Nowadays Vassiliev invariants are defined as follows.  Let us denote by $\mathbfcal X_r$ the set of ambient isotopy classes of singular links with exactly $r$ double points.
For an abelian group $A$, and for an $A$-valued link invariant $v:\mathbfcal X_0\to A$, the \emph{$r$-th derivative} $v^{(r)}$ of $v$ is a map $\mathbfcal X_r\to A$ defined inductively by $v^{(0)}=v$ and
\begin{equation}
\label{eq:V-skein}
v^{(r+1)}\left(\diagCrossSingUp{}\right)
=
v^{(r)}\left(\diagCrossPosUp{}\right)
-
v^{(r)}\left(\diagCrossNegUp{}\right)
\quad.
\end{equation}
The equation \eqref{eq:V-skein} is called \emph{Vassiliev skein relation} and 
the map $v^{(r)}$ is called the \emph{$r$-th derivative} of $v$.
If $v^{(r)}=0$ while $v^{(r-1)}\neq 0$, we call $v$ a Vassiliev invariant of order $r$.
The formula \eqref{eq:V-skein} is natural from the following point of view: singular knots are objects that keep track of crossing changes in a passage from one knot to a neighbor in a stratified space $\bigcup_j\left(\mathcal{M}_j-\Sigma_j\right)$, where $\mathcal{M}_j$ (resp. $\Sigma_j$) is the space of immersions $S^1\to S^3$ with exactly $j$ transverse double points (resp. with at least $j$ double points and other singularities).

Vassiliev invariants and the Jones polynomial are related via Taylor expansion.
More precisely, let $\widetilde V$ be the unnormalized Jones polynomial, with variable $q$, and
\[
U_x (L) = \sum_{n=0}^{\infty} v_n (L) x^n,    
\]
the Taylor expansion of $\left.\widetilde V(L)\right|_{q=e^x}$.
Then, each coefficient $v_n (L)$ is a Vassiliev invariant of order $n$ (cf.~Birman-Lin \cite{BirmanLin1993}).

In 2000, Khovanov \cite{Khovanov2000} obtained a categorification of the Jones  polynomial; namely, he constructed a bi-graded abelian group $\mathit{Kh}^{\ast,\star}(D;M)$ for each diagram $D$ and an abelian group $M$ such that
\begin{enumerate}[label=\upshape(\arabic*)]
  \item $\mathit{Kh}^{\ast,\star}(D;M)$ is invariant under Reidemeister moves so that it defines an invariant for links;
  \item if $k$ is a field, the graded Euler characteristic
\[
\left[\mathit{Kh}(L;k)\right]_q = \sum_{i, j} (-1)^i q^j \dim_{k}\mathit{Kh}^{i, j}(L;k)
\]
equals the unnormalized Jones polynomial $\widetilde V(L)$.
\end{enumerate}
The construction is deeply involved with $2$-dimensional topological quantum field theory.
In fact, Bar-Natan \cite{BarNatan2005} generalized the construction of Khovanov in terms of a category of cobordisms.

Our approach to crossing change on Khovanov homology is as follows: we consider a cobordism of the form
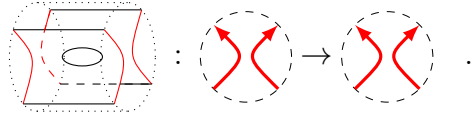
\[
\cobHoleVV{-.5ex}{.4}
\;:\;
\diagSmoothUp
\to \diagSmoothUp
\quad.
\]
It will turn out that applying it around a negative crossing, we get a chain map
\begin{equation}
\label{eq:intro:genu1}
\widehat\Phi:\;
C\left(\diagCrossNegUp{}\right)\otimes \mathbb F_2
\to
C\left(\diagCrossPosUp{}\right)\otimes \mathbb F_2
\quad,
\end{equation}
which we call the \emph{genus-$1$ map}.
With regard to Vassiliev skein relation, we define a complex $C(D;\mathbb F_2)$ for singular link diagram $D$ by taking mapping cones of $\widehat\Phi$ recursively.
Set $\mathit{Kh}(D;\mathbb F_2)$ to be the homology of $C(D;\mathbb F_2)$ (\cref{def:Kh-sing}), then our first main result is described as follows.

\begin{mainthm}
\label{main:KHinv}
Regarding Khovanov homology, we have the following isomorphisms:
\begin{gather}
\mathit{Kh}^{i,j}\left(\diagCrossSingRivOL;\mathbb{F}_2\right)
\cong \mathit{Kh}^{i,j}\left(\diagCrossSingRivOR;\mathbb{F}_2 \right)
\quad,\\
\mathit{Kh}^{i,j}\left(\diagCrossSingRivUL; \mathbb{F}_2 \right)
\cong \mathit{Kh}^{i,j} \left(\diagCrossSingRivUR; \mathbb{F}_2 \right)
\quad,\\
\mathit{Kh}^{i,j} \left(\diagRvSingL; \mathbb{F}_2 \right)
\cong \mathit{Kh}^{i,j} \left(\diagRvSingR; \mathbb{F}_2 \right)
\quad.
\end{gather}
Consequently, $Kh(\blank;\mathbb{F}_2)$ is an invariant of singular links.
\end{mainthm}
Theorem~\ref{main:catVas} gives a ``categorified Vassiliev skein relation".  
\begin{mainthm}\label{main:catVas}
Let $L$ be a singular link with a double point $b$, and put $L_+$ and $L_-$ to be the (possibly singular) links obtained from $L$ by resolving $b$ to positive and negative crossings respectively.
Then, there is a long exact sequence
\[
\cdots
\to\mathit{Kh}^{i-1,j}(L;\mathbb F_2)
\to \mathit{Kh}^{i,j}(L_-;\mathbb F_2)
\xrightarrow{\widehat\Phi} \mathit{Kh}^{i,j}(L_+;\mathbb F_2)
\to\mathit{Kh}^{i,j}(L;\mathbb F_2)
\to\cdots
\]
for each $j\in\mathbb Z$, where $\widehat\Phi$ is the genus-$1$ map applying around $b$.
\end{mainthm}

\begin{maincor}\label{cor:KHsing-Jones}
For a singular link $L$ with exactly $r$ double points, we have
\[
\left[\mathit{Kh}(L;\mathbb F_2)\right]_q
= \widetilde V^{(r)}(L)
\]
here $\widetilde V^{(r)}$ is the $r$-th derivative of the unnormalized Jones polynomial $\widetilde V:\mathbfcal X_r\to\mathbb Z[q,q^{-1}]$.
\end{maincor}

As it is well-known, Kontsevich \cite{Kontsevich1993} obtained a sequence of Vassiliev invariants whose target space, denoted by $\mathcal{A} (S^1)$, generated by chord diagrams.   
One of the relations is called the FI relation: every chord diagram having an isolated chord vanishes.  It means that every Vassiliev invariant $v$ of order $n$ satisfies that $v (L)$ $=$ $0$ for any singular link $L$ one of whose double points is of the form (\ref{eq:KH-FIrel:FIsing}).  We categorify it and obtain the following result.       
\begin{mainthm}
\label{main:KH-FIrel}
Khovanov homology satisfies the FI relation.
More precisely, suppose $L$ is a singular link one of whose double points is of the form
\begin{equation}
\label{eq:KH-FIrel:FIsing}
\diagFiSing
\quad.
\end{equation}
Then, we have $\mathit{Kh}^{i,j}(L;\mathbb F_2)=0$ for all integers $i$ and $j$.
\end{mainthm}

The plan of this paper is as follows.
After recalling the definition of Khovanov homology for ordinary links in \cref{sec:quick-review}, we introduce the genus-$1$ map \eqref{eq:intro:genu1} in \cref{sec:crs-Kh}.
Using this, we build up a chain complex $C(D;\mathbb F_2)$ for singular link diagrams $D$ in \cref{sec:KHsing}.
The basic tools are cofiber sequences and their \emph{cubes} \cite{Steiner1986, BrownEllis1988}; we give a brief introduction for them for the convenience of the readers.
In \cref{sec:invariance}, we show invariances of $\HP$ to prove \cref{main:KHinv}.  
Finally, in \cref{sec:prop}, we investigate basic properties of Khovanov homology on singular links.
\Cref{main:catVas} and \cref{main:KH-FIrel} will be proved there.

\section{Quick review of Khovanov homology}
\label{sec:quick-review}

We quickly review the construction of Khovanov homology for ordinary links.
\subsection{TQFT and Frobenius algebras}
\label{sec:quick-review:tqft}
%
%
%
\begin{definition}
Let $\mathbf{Cob}_2$ be the $2$-dimensional cobordism category.  
Let $k$ be a commutative ring.
Then, a \emph{$2$-dimensional topological quantum field theory} (\emph{TQFT} for short) is a symmetric monoidal functor
\[
Z:\mathbf{Cob}_2\to\mathbf{Mod}_k
\ ,
\]
where $\mathbf{Mod}_k$ is the category of $k$-modules and $k$-homomorphisms.
\end{definition}
A is well-known, a Frobenius algebra $A$ gives rise to a $2$-dimensional TQFT $Z_A$ with  
$Z_A(S^1) = A$.
The following are examples of TQFT operations $Z_A$: 
\[
\begin{array}{rcc}
  Z_A\left(
  \tikz[scale=.4, baseline=-.5ex]{
  \coordinate (Ai1) at (-2,1.7);
  \draw (Ai1) +(90:.5 and 1) arc (90:270:.5 and 1);
  \draw[dashed] (Ai1) +(-90:.5 and 1) arc (-90:90:.5 and 1);
  \coordinate (Ai2) at (-2,-1.7);
  \draw (Ai2) +(90:.5 and 1) arc (90:270:.5 and 1);
  \draw[dashed] (Ai2) +(-90:.5 and 1) arc (-90:90:.5 and 1);
  \coordinate (Ao) at (2,0);
  \draw (Ao) ellipse (.5 and 1);
  \draw ($(Ai1)+(0,-1)$) .. controls +(1.8,0) and +(1.8,0) .. ($(Ai2)+(0,1)$);
  \draw ($(Ai1)+(0,1)$) .. controls +(2,0) and +(-2,0) .. ($(Ao)+(0,1)$);
  \draw ($(Ai2)+(0,-1)$) .. controls +(2,0) and +(-2,0) .. ($(Ao)+(0,-1)$);
  }
  \right)
  = \mu
  & :A\otimes_k A\to A\ ;
  & a_1\otimes a_2\mapsto a_1a_2, 
  \\[9ex]
  Z_A\left(
  \tikz[scale=.4, baseline=-.5ex]{
  \coordinate (Ai) at (-2,0);
  \draw (Ai) +(90:.5 and 1) arc (90:270:.5 and 1);
  \draw[dashed] (Ai) +(-90:.5 and 1) arc (-90:90:.5 and 1);
  \coordinate (Ao1) at (2,1.7);
  \draw (Ao1) ellipse (.5 and 1);
  \coordinate (Ao2) at (2,-1.7);
  \draw (Ao2) ellipse (.5 and 1);
  \draw ($(Ai)+(0,1)$) .. controls +(2,0) and +(-2,0) .. ($(Ao1)+(0,1)$);
  \draw ($(Ai)+(0,-1)$) .. controls +(2,0) and +(-2,0) .. ($(Ao2)+(0,-1)$);
  \draw ($(Ao1)+(0,-1)$) .. controls +(-1.8,0) and +(-1.8,0) .. ($(Ao2)+(0,1)$);
  }
  \right)
  = \Delta
  & :A\to A\otimes_k A\ ;
  & a\mapsto a^{(1)}\otimes a^{(2)}.  
\end{array}
\]
We here use Sweedler notation for the comultiplication.
In this paper, we mainly consider the Frobenius algebra  $A=\mathbb Z[x]/(x^2)$ with   
$\begin{gathered}
\Delta(1)=x\otimes 1 + 1\otimes x, \Delta(x)=x\otimes x.    
\end{gathered}$
In this case, we give a grading on $Z_A (S^1) \cong A$ as $\deg 1=1$ and $\deg x=-1$.
%
\subsection{Khovanov complex
}
\label{sec:quick-review:grmod}

For a link $L\subset S^3$, fix its diagram $D\subset\mathbb R^2$, and put $c(D)$ the set of crossings of $D$.
Each subset $s\subset c(D)$ is called \emph{state} on $D$; we write $|s|$ the cardinality.    
For each state $s$ on $D$, we write $D_s$ the closed $1$-dimensional manifold obtained by \emph{smoothing} all the crossings of $D$ according to $s$:
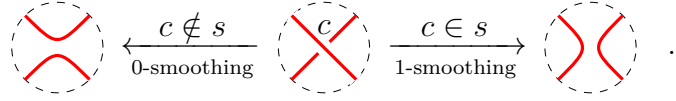
\[
\diagSmoothH
\;\xleftarrow[\text{$0$-smoothing}]{\displaystyle c\notin s}\;
\diagCrossNeg{c}
\;\xrightarrow[\text{$1$-smoothing}]{\displaystyle c\in s}\;
\diagSmoothV
\quad.
\]
For each integer $i$, define a graded abelian group $\overbar C^i(D)$ by
\[
\overbar C^{i,j}(D)
\coloneqq \bigoplus_{|s|=i} Z_A(D_s)^{j-i}
\ ,
\]
here, $Z_A(D_s)^r$ is the homogeneous component of degree $r$ of the graded abelian group $Z_A (D_s)$.    
Then, for each $c \notin s$, we write $\delta_c$  the map corresponding to a cobordism from a $0$-smoothing to a $1$-smoothing for some $j$:  
\[ 
\delta_c := Z_A \left( \cobSaddleHV{-.5ex}{.4}\ \right): Z_A \left( \diagSmoothH \right)^{j}
\to Z_A \left( \diagSmoothV \right)^{j-1}
\ .
\]  
The differential $d$ is defined as follows: 
\[
d=\sum_{c\in c(D)}\pm\delta_c:\overbar C^{i,j}(D)\to\overbar C^{i+1,j}(D)
\ ,
\]
here we assume $\delta_c=0$ on $Z_A(D_s)$ with $c\in s$, 
and for the definitions of signs of $\delta_c$, the reader can refer to \cite{Khovanov2000} and \cite{BarNatan2005}.    
The resulting complex is called the \emph{unshifted Khovanov complex} of $D$.  
Considering the orientation on the diagram $D$, let us denote by $n_+$ and by $n_-$ the number of positive and negative crossings respectively.
We define the \emph{Khovanov complex} $C(D)$ by
\[
C^{i,j}(D)
\coloneqq \overbar C^{i+n_-,\ j+2n_- -n_+}(D)
\]
with the same differential as $\overbar C(D)$.
We set
\[
\mathit{Kh}^{i,j}(D)
\coloneqq H^i(C^{\ast,j}(D))
\]
and call it the \emph{Khovanov homology of $D$}.
One may also consider the version with coefficients as
\[
\mathit{Kh}^{i,j}(D;A)
\coloneqq H^i(C^{\ast,j}(D)\otimes A)
\]
for each abelian group $A$; we call it the \emph{Khovanov homology of $D$ with coefficients in $A$}.
Note that there is a canonical identification $\mathit{Kh}^{i,j}(D)=\mathit{Kh}^{i,j}(D;\mathbb Z)$.

\begin{theorem}[Khovanov \cite{Khovanov2000}]
\label{theo:Khovanov}
Let $L$ be an oriented link and $D$ a diagram of $L$.
\begin{enumerate}[label=\upshape(\arabic*)]
  \item If $D'$ is a diagram which is connected to $D$ by a sequence of Reidemeister moves, then there is a quasi-isomorphisms $C(D)\to C(D')$.
In particular, for every abelian group $A$, $\mathit{Kh}^{i,j}(D;A)$ is an invariant for $L$, so it is safe to write $\mathit{Kh}^{i,j}(L;A)\coloneqq\mathit{Kh}^{i,j}(D;A)$.
  \item For any field $k$, the formal sum
\[
\left[\mathit{Kh}(L;k)\right]_q
\coloneqq \sum_{i,j} (-1)^iq^j\dim_k \mathit{Kh}^{i,j}(L;k)
\]
equals the unnormalized Jones polynomial of $L$.
\end{enumerate}
\end{theorem}
%
%
\begin{notation}
\label{note:shiftop}
With regard to shifting on (bi)graded modules, it is convenient to use shift operators.
Namely, if $W=\{W^{i,j}\}_{i,j}$ is a bigraded module, then we define a bigraded module $W[k]\{l\}$ by
\[
W[k]\{l\}^{i,j}
\coloneqq W^{i-k,j-l}
\ .
\]
For example, we have $C(D)=\overbar C(D)[-n_-]\{-2n_- + n_+\}$ for a link diagram $D$.
\end{notation}
\subsection{Enhanced states}
\label{sec:quick-review:enhst}

To investigate Khovanov complexes, it is convenient to describe them in terms of bases.
The following notion was well-known in the study of Kauffman brackets.

\begin{definition}
Let $D$ be a link diagram with the set $c(D)$ of crossings.
Then, an \emph{enhanced state} on $D$ is a pair $(s,\rho)$ of a state $s\subset c(D)$ and a map
\[
\rho:\pi_0(D_s)\to \{1,x\}\ ,
\]
here $\pi_0(D_s)$ is the set of connected components of the $1$-manifold $D_s$.
\end{definition}

For an enhanced state $(s,\rho)$, we define its \emph{degree} by
\[
\deg(s,\rho)
\coloneqq \#\{\alpha\in\pi_0(D_s)\mid\rho(\alpha)=1\} - \#\{\alpha\in\pi_0(D_s)\mid\rho(\alpha)=x\}
\ .
\]
To a given enhanced state $(s,\rho)$ on $D$, we assign an element
\[
v_{(s,\rho)}
\coloneqq \bigoplus_{\alpha\in\pi_0(D_s)}\rho(\alpha)
\in Z_A(D_s)^{\deg(s,\rho)}
= Z_A(D_s)\{|s|\}^{\deg(s,\rho)+|s|}
\subset \overbar C^{|s|,\deg(s,\rho)+|s|}(D)
\quad.
\]
The following is obvious.

\begin{lemma}
\label{lem:KH-basisEnh}
Let $D$ be a diagram.
Then, for each pair $(i,j)$ of integers, the set
\[
\left\{v_{(s,\rho)}\;\middle|\;\text{$(s,\rho)$: enhanced state with $|s|=i$ and $\deg(s,\rho)=j-i$}\right\}
\]
forms a basis of the abelian group $\overbar C^{i,j}(D)$.
\end{lemma}

\begin{remark}
The use of enhanced states for Khovanov homology is due to Viro \cite{Viro2004}.
\end{remark}

We depict an enhanced state $(s,\rho)$ on the diagram $D$ in the following way:
\begin{itemize}
  \item for each crossing $c\in c(D)$, draw a line segment in the direction of the smoothing applied in $D_s$:
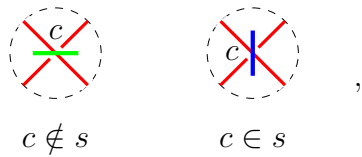
\[
\begin{array}{ccc}
  \tikz{\node at (0,0) {\useTikzBox{diagCrossNeg}}; \draw[green,ultra thick] (-.5,0) -- (.5,0); \node at (0,1.5ex) {$c$}} &\quad& \tikz{\node at (0,0) {\useTikzBox{diagCrossNeg}}; \draw[blue,ultra thick] (0,-.5) -- (0,.5); \node at (-1.5ex,0) {$c$}} \\
 c \notin s &\quad& c \in s
\end{array}
\quad,
\]
here we use different colors to indicate different states;
  \item attach the label $\rho(\alpha)$ to a circle of $D$ such that the corresponding circle  in $D_s$ is a part of $\alpha$.
\end{itemize}

\section{Crossing change on Khovanov complexes}
\label{sec:crs-Kh}

Let $D$ be a diagram of a singular link with a double point $b$, and put $D_+$ and $D_-$ the diagrams obtained from $D$ by resolving $b$ into a positive and negative crossing respectively.
For a link invariant $v$, Vassiliev skein relation \eqref{eq:V-skein} asserts that the value $v(D)$ measures the effect of crossing change at $b$ on $v$.
In this section, we discuss an analogue on Khovanov homology.
More precisely, we define a chain map $\widehat\Phi:C(D_-)\to C(D_+)$ and see how this categorifies the subtraction in Vassiliev skein relation.

\subsection{Mapping cones as chain-level subtraction}
\label{sec:cone}

We begin with a quick review on mapping cone construction on chain complexes.

Let $\mathcal A$ be an abelian category.
Recall that, if $f:X\to Y$ is a chain map between chain complexes in $\mathcal A$, then the \emph{mapping cone} $\operatorname{Cone}(f)$ is a chain complex with a short exact sequence
\[
0 \to Y\xrightarrow{i_Y} \operatorname{Cone}(f)\xrightarrow{p_X} X[-1]\to 0
\]
so that the following sequence is exact:
\[
\cdots
\xrightarrow{(p_X)_\ast} H^i(X)
\xrightarrow{f_\ast} H^i(Y)
\xrightarrow{(i_Y)_\ast} H^i(\operatorname{Cone}(f))
\xrightarrow{(p_X)_\ast} H^{i+1}(X)
\to \cdots
\quad.
\]
Specifically, in case where $\mathcal A$ is the category of vector spaces over a field $k$, we obtain the following identity on the Euler characteristics as long as they make sense:
\begin{equation}
\label{eq:exact-Euler}
\chi(X)-\chi(Y)+\chi(\operatorname{Cone}(f))=0
\ .
\end{equation}
In this point of view, we may say $\operatorname{Cone}(f)$ \emph{categorifies} the subtraction $\chi(Y)-\chi(X)$.

Actually, $\operatorname{Cone}(f)$ plays an important roles on the homotopy theory on chain complexes.
The following are the properties which characterize $\operatorname{Cone}(f)$ as a \emph{homotopy cofiber} of $f$.

\begin{lemma}[{\cite[Proposition~3.1.3]{Verdier1996}}]
\label{lem:cone=cofib}
Suppose we have a sequence
\[
X\xrightarrow{f} Y\xrightarrow{g} Z
\]
of chain maps in an abelian category $\mathcal A$.
Then, there is a one-to-one correspondence between the following data:
\begin{enumerate}[label=\upshape(\alph*)]
  \item\label{sub:cone=cofib:null} chain homotopies from the composition $gf$ to $0$;
  \item\label{sub:cone=cofib:cone} chain maps $\widehat g:\operatorname{Cone}(f)\to Z$ which makes the diagram below commute:
\[
\begin{tikzcd}
Y \ar[r,"g"] \ar[d,hook,"i_Y"']& Z \\
\operatorname{Cone}(f) \ar[ur,"\widehat g"'] &
\end{tikzcd}
\quad.
\]
\end{enumerate}
\end{lemma}

This motivates us to introduce the following notion.

\begin{definition}
\label{def:cofib-seq}
Let $\mathcal A$ be an abelian category.
Suppose we have a sequence of chain complexes
\begin{equation}
\label{eq:def:cofib-seq:seq}
X\xrightarrow{f} Y\xrightarrow{g} Z
\end{equation}
in $\mathcal A$ together with a chain homotopy $\Psi:gf\sim 0$.
Then, we call it a \emph{cofiber sequence} if the induced map $\operatorname{Cone}(f)\to Z$ is a quasi-isomorphism; i.e. it induces isomorphisms on homology groups.
In this case, we call the map $Y\to Z$ (or even $Z$ itself by abuse of notation) a \emph{homotopy cofiber} of $f$.
\end{definition}

\begin{remark}
The reader should note that, for a sequence \eqref{eq:def:cofib-seq:seq}, whether the induced map $\operatorname{Cone}(f)\to Z$ is a quasi-isomorphism does depend on a chain homotopy $\Psi$.
\end{remark}

\begin{example}
If $X$ is a chain complex with a subcomplex $X'\subset X$, then the sequence
\begin{equation}
\label{eq:quot=cofib}
X'\overset{i}\hookrightarrow X\overset{p}\twoheadrightarrow X/X'
\end{equation}
is a cofiber sequence.
\end{example}

\begin{example}
Let $f:X\to Y$ be a quasi-isomorphism between chain complexes.
Then, a complex $Z$ is a cofiber of $f$ if and only if all the homology groups $H^\ast(Z)$ vanish.
\end{example}

The following results follow immediately from \cref{def:cofib-seq}.

\begin{lemma}
\label{lem:cofib-longex}
For every cofiber sequence $X\xrightarrow{f}Y\xrightarrow{g} Z$, there is a map $\delta:H^\ast(Z)\to H^{\ast+1}(X)$ so that the sequence below is exact:
\[
\cdots
\to H^i(X)
\xrightarrow{f_\ast} H^i(Y)
\xrightarrow{g_\ast} H^i(Z)
\xrightarrow{\delta} H^{i+1}(X)
\xrightarrow{f_\ast} H^{i+1}(Y)
\to \cdots
\quad.
\]
\end{lemma}

\begin{corollary}
\label{cor:cofibseq-Euler}
If there is a cofiber sequence $X\to Y\to Z$ of chain complexes over a field $k$, then we have
\[
\chi(X)-\chi(Y)+\chi(Z)=0
\]
provided the homology groups $H^\ast(X)$, $H^\ast(Y)$, and $H^\ast(Z)$ are all finite dimensional.
\end{corollary}

Finally, the homotopy invariance of cofiber sequences is sated as follows.
The proof will be found in \cite{Weibel1994} for example.

\begin{proposition}
\label{prop:hinv-cofib}
Suppose we have a commutative diagram
\begin{equation}
\label{eq:hinv-cofib:univ}
\begin{tikzcd}
X' \ar[r,"f'"] \ar[d,"u"'] & Y' \ar[d,"v"] \ar[r,"h'"] & W' \ar[d,"w"] \\
X \ar[r,"f"] & Y \ar[r,"h"] & W
\end{tikzcd}
\end{equation}
of chain maps.
Suppose in addition that there are chain homotopies $\Psi:gf\sim 0$ and $\Psi':g'f'\sim 0$ such that $\Psi u =w\Psi'$.
Then, for every cofiber $Y\xrightarrow{g} Z$ of $f$ which factors through the map $h$ in \eqref{eq:hinv-cofib:univ}, there is a commutative square:
\begin{equation}
\label{eq:hinv-cofib:ind}
\begin{tikzcd}
\operatorname{Cone}(f') \ar[r,"\widehat h'"] \ar[d] & W' \ar[d,"s"] \\
Z \ar[r] & W
\end{tikzcd}
\quad.
\end{equation}
Furthermore, the vertical arrows in \eqref{eq:hinv-cofib:ind} are quasi-isomorphisms as soon as so are those in \eqref{eq:hinv-cofib:univ}.
\end{proposition}

\begin{remark}
\label{rem:cone=fib}
All the arguments above are \emph{dualizable}.
In particular, for a sequence $X\xrightarrow{f}Y\xrightarrow{g}Z$ of chain maps, the following data are equivalent (cf. \cref{lem:cone=cofib}):
\begin{enumerate}[label=\upshape(\alph*)]
  \item a chain homotopy from $gf$ to $0$;
  \item a chain map $\overbar f:X\to\operatorname{Cone}(g)[1]$ which makes the diagram below commute:
\[
\begin{tikzcd}
& \operatorname{Cone}(g)[1] \ar[d,two heads,"p_Y"] \\
X \ar[r,"f"] \ar[ur,"\overbar f"] & Y
\end{tikzcd}
\quad.
\]
\end{enumerate}
\end{remark}

\subsection{Genus-$1$ map}
\label{sec:genus-1}

We now define a chain map
\[
\widehat\Phi:
C\left(\diagCrossNegUp{}\right)
\to C\left(\diagCrossPosUp{}\right)
\quad.
\]
We first consider the chain map induced by the following cobordism:
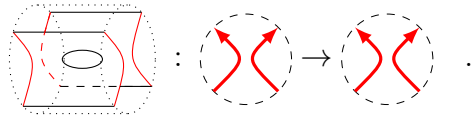
\[
\cobHoleVV{-.5ex}{.4}
\;:\;
\diagSmoothUp
\to \diagSmoothUp
\quad.
\]
We denote by $\Phi$ the induced map on Khovaonv complexes.
On the other hand, we also have the following single saddle operations:
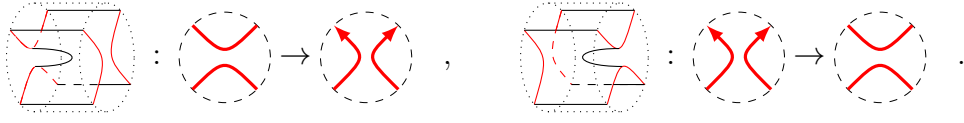
\[
\cobSaddleHV{-.5ex}{.4}
\;:\;
\diagSmoothH
\to \diagSmoothUp
\quad,\qquad
\cobSaddleVH{-.5ex}{.4}
\;:\;
\diagSmoothUp
\to \diagSmoothH
\quad.
\]
We denote by $\delta_-$ and $\delta_+$ respectively the map induced on Khovanov complexes.
In terms of the unshifted Khovanov complexes, we obtain the sequence of chain maps below:
\begin{equation}
\label{eq:delta-Phi-delta}
\overbar C\left(\scalebox{0.8}{\diagSmoothH}\right)
\xrightarrow{\delta_-}
\overbar C\left(\scalebox{0.8}{\diagSmoothUp}\right)\{1\}
\xrightarrow{\Phi}
\overbar C\left(\scalebox{0.8}{\diagSmoothUp}\right)\{3\}
\xrightarrow{\delta_+\{3\}}
\overbar C\left(\scalebox{0.8}{\diagSmoothH}\right)\{4\}
\end{equation}
where all maps are of bidegree $(0,0)$ thanks to the degree shifts.

\begin{proposition}
\label{prop:delta-Phi}
In the situation above, the compositions $\Phi\delta_-$ and $\delta_+\Phi$ are zero after tensored with the field $\mathbb F_2$.
\end{proposition}
\begin{proof}
Put $A=\mathbb Z[x]/(x^2)$ and regard it as a graded abelian group as in \cref{sec:quick-review:tqft}.
Note that both maps $\Phi\delta_-$ and $\delta_+\Phi$ are constructed by applying the TQFT $Z_A$ on cobordisms of three saddles.
Hence, to verify the statement, it suffices to show that both of the compositions
\[
\mu\circ\Delta\circ\mu:A\otimes A\to A
\ ,\quad
\Delta\circ\mu\circ\Delta:A\to A\otimes A
\]
vanish in characteristic $2$, where $\mu$ and $\Delta$ are the multiplication and the comultiplication respectively.
This is verified by the direct computations.
\end{proof}

\begin{corollary}
\label{cor:Phi-hat}
The sequence \eqref{eq:delta-Phi-delta} induces a chain map
\[
\widehat\Phi:\;
C\left(\diagCrossNegUp{}\right)\otimes \mathbb F_2
\to
C\left(\diagCrossPosUp{}\right)\otimes \mathbb F_2
\]
of bidegree $(0,0)$.
\end{corollary}
\begin{proof}
Note that we have canonical identifications
\[
\operatorname{Cone}(\delta_-)
= \overbar C\left(\diagCrossNegUp{}\right)[-1]
\ ,\quad
\operatorname{Cone}(\delta_+)
= \overbar C\left(\diagCrossPosUp{}\right)[-1]
\ .
\]
Hence, in view of \cref{lem:cone=cofib} and \cref{rem:cone=fib}, \cref{prop:delta-Phi} yields a chain map $\widehat\Phi$ as required.
As for its bidegree, the result follows from the direct computation of the numbers of positive and negative crossings on the two diagrams.
\end{proof}

In what follows, the map $\widehat\Phi$ is referred to as the \emph{genus-$1$} map.
In terms of enhanced states (see \cref{sec:quick-review:enhst}), we have an explicit description.
Namely, it is given by the following local formula:
\[
\begin{array}{cccccc}
\tikz[baseline=-.5ex]{%
  \node at (0,0) {\arrayHCrossNeg};
  \draw[green,ultra thick] (-.5,0) -- node[above,black] {$p$} (.5,0);
} &
  \xmapsto{\widehat\Phi} &
  0\quad,\quad&
\tikz[baseline=-.5ex]{%
  \node at (0,0) {\arrayHCrossNeg};
  \draw[blue,ultra thick] (0,.5) -- (0,-.5);
  \node at (-.5,0) {$p$};
  \node at (.5,0) {$q$};
} &
  \xmapsto{\widehat\Phi}&
\tikz[baseline=-.5ex]{%
  \node at (0,0) {\arrayHCrossPos};
  \draw[green,ultra thick] (0,.5) -- (0,-.5);
  \node at (-1,0) {$(pq)^{(1)}$};
  \node at (1,0) {$(pq)^{(2)}$};
}
\\
\tikz[baseline=-.5ex]{%
  \node at (0,0) {\arrayVCrossPos};
  \draw[green,ultra thick] (-.5,0) -- (.5,0);
  \node at (0,.5) {$p$};
  \node at (0,-.5) {$q$};
} &
  \mapsto&
  0\quad,\quad &
\tikz[baseline=-.5ex]{%
  \node at (0,0) {\arrayVCrossNeg};
  \draw[blue,ultra thick] (0,.5) -- node[right,black] {$p$} (0,-.5);
} &
  \mapsto &
\tikz[baseline=-.5ex]{%
  \node at (0,0) {\arrayVCrossNeg};
  \draw[green,ultra thick] (0,.5) -- node[right,black] {$p^{(1)}p^{(2)}$} (0,-.5);
}
\end{array}
\quad,
\]
here we use \emph{Sweedler notation} for comultiplications in the Frobenius algebra structure.

We end the section by mentioning the relation to Vassiliev skein relation.
Since the genus-$1$ map $\widehat\Phi$ is of bidegree $(0,0)$, by virtue of \cref{lem:cofib-longex}, we obtain an equation
\[
\left[ H^\ast(\operatorname{Cone}(\widehat\Phi))\right]_q
= \left[\mathit{Kh}\left(\diagCrossPosUp{}\right)\right]_q
- \left[\mathit{Kh}\left(\diagCrossNegUp{}\right)\right]_q
\quad.
\]
This is why we regard $\widehat\Phi$ as a chain-level realization of \emph{crossing change} in Vassiliev skein relation.
Actually, in this point of view, we extend Khovanov homology to singular links in \cref{sec:KHsing}.

\section{Khovanov homology of singular knots and links}
\label{sec:KHsing}

In the previous sections, we saw Vassiliev skein relation is categorified as the mapping cone of the genus-$1$ map $\widehat\Phi$.
The goal of the section is to show this actually gives rise to an extension of Khovanov homology to singular links.

\subsection{Multiple mapping cone of cubes}
\label{sec:KHsing:mult-cone}

To give an explicit description for Khovanov homology on singular knots, we use the theory of \emph{multiple mapping cone}, which was first developed in \cite{BrownEllis1988} as the dual of the results in \cite{Steiner1986}.

We concentrate on the case of chain complexes though the original one was discussed in the case of topological spaces.
More precisely, throughout the section, we fix an abelian category $\mathcal A$ and discuss the category $\mathbf{Ch}(\mathcal A)$ of chain complexes in $\mathcal A$.
In the later section, we take $\mathcal A$ to be the category of graded modules to apply the theory to Khovanov homology.

\begin{definition}
\label{def:n-cube-cplx}
For a non-negative integer $n$, an \emph{$n$-cube of complexes} in $\mathcal A$ is a family $X=\{X_A\}_A$ indexed by subsets $A\subset\langle n\rangle$ together with chain maps $\varphi_s:X_A\to X_{A\cup\{s\}}$ for $s\in\langle n\rangle\setminus A$ such that the square below commutes for each $s,t\in\langle n\rangle\setminus A$ with $s\neq t$:
\[
\begin{tikzcd}
X_A \ar[r,"\varphi_s"] \ar[d,"\varphi_t"] & X_{A\cup \{s\}} \ar[d,"\varphi_t"] \\
X_{A\cup\{t\}} \ar[r,"\varphi_s"] & X_{A\cup\{s,t\}}
\end{tikzcd}
\quad.
\]
\end{definition}

\begin{definition}
\label{def:n-cube-cofib}
An \emph{$n$-cube of cofiber sequences} in $\mathbf{Ch}(\mathcal A)$ consists of
\begin{itemize}
  \item a family $Y=\{Y_{B,A}\}_{(B,A)}$ of chain complexes in $\mathcal A$ indexed by pairs of disjoint subsets $A\amalg B\subset\langle n\rangle$;
  \item a cofiber sequence
\[
Y_{B,A}
\xrightarrow{\varphi_i} Y_{B,A\cup\{i\}}
\xrightarrow{\psi_i} Y_{B\cup\{i\},A}
\]
for each $i\in\langle n\rangle\setminus (A\amalg B)$;
\end{itemize}
such that the diagram below commutes
\[
\begin{tikzcd}
Y_{B,A} \ar[r,"\varphi_i"] \ar[d,"\varphi_j"] & Y_{B,A\cup\{i\}} \ar[r,"\psi_i"] \ar[d,"\varphi_j"] & Y_{B\cup\{i\},A} \ar[d,"\varphi_j"] \\
Y_{B,A\cup\{j\}} \ar[r,"\varphi_i"] \ar[d,"\psi_j"] & Y_{B,A\cup\{i,j\}} \ar[r,"\psi_i"] \ar[d,"\psi_j"] & Y_{B\cup\{i\},A\cup\{j\}} \ar[d,"\psi_j"] \\
Y_{B\cup\{j\},A} \ar[r,"\varphi_i"] & Y_{B\cup\{j\},A\cup\{i\}} \ar[r,"\psi_i"] & Y_{B\cup\{i,j\},A}
\end{tikzcd}
\]
for each distinct elements $i,j\in\langle n\rangle\setminus (A\amalg B)$.
\end{definition}

Note that if $Y=\{Y_{B,A}\}_{(B,A)}$ is an $n$-cube of cofiber sequences in $\mathbf{Ch}(\mathcal A)$, the subfamily $\underline Y\coloneqq\{Y_{\varnothing,A}\}_A$ forms an $n$-cube of complexes.
We are interested in the inverse direction: we want to extend an $n$-cube of complexes to an $n$-cube of cofiber sequences.

Given an $n$-cube $X$ of complexes, we define a chain complex $\widetilde X_{B,A}$ so that
\begin{itemize}
  \item the underlying graded object is given by
\[
\widetilde X_{B,A}
\coloneqq \bigoplus_{S\subset B} X_{A\cup S}\bigl[-|B\setminus S|\bigr]
\ ,
\]
where $|\blank|$ denotes the cardinality;
  \item the differential $d^i_{B,A}:\widetilde X^i_{B,A}\to \widetilde X^{i+1}_{B,A}$ is componentwisely given as
\[
d_{B,A}^i:
X_{A\cup S}\bigl[-|B\setminus S|\bigr]^i
=X_{A\cup S}^{i+|B\setminus S|}
\xrightarrow{d_{B,A}^{i;S,T}} X_{A\cup T}^{i+|B\setminus T|+1}
=X_{A\cup T}\bigl[-|B\setminus T|\bigr]^{i+1}
\]
with
\[
d_{B,A}^{i;S,T} =
\begin{cases}
(-1)^{|B\setminus S|}d_{A\cup S}^{i+|B\setminus S|} &\quad \text{if $S=T$}\ , \\[1.5ex]
(-1)^{\nu(t;S)}\varphi_t^{i+|B\setminus S|} &\quad \text{if $T=S\cup\{t\}$ with $t\notin S$}\ ,\\[1.5ex]
0 &\quad \text{otherwise}\ ,
\end{cases}
\]
here $\nu(t;S)\coloneqq\#\{s\in S\mid s>t\}$.
\end{itemize}
It is easily checked that $(\widetilde X_{B,A},d_{B,A})$ is a chain complex in $\mathcal A$; in particular $d_{B,A}\circ d_{B,A}=0$.
Furthermore, for each $s\in\langle n\rangle\setminus(A\amalg B)$, we have a canonical decomposition
\begin{multline}
\label{eq:multcone-decomp}
\widetilde X_{B\cup\{s\},A}
= \left(\bigoplus_{S\subset B} X_{A\cup\{s\}\cup S}\bigl[-|B\setminus S|\bigr]\right)
\oplus \left(\bigoplus_{S\subset B} X_{A\cup S}\bigl[-|B\setminus S|\bigr]\right)[-1] \\
= \widetilde X_{B,A\cup\{s\}}\oplus \widetilde X_{B,A}[-1]
\mathrlap{\quad.}
\end{multline}
With regard to the decomposition, the differential $d_{B\cup\{s\},A}$ on $\widetilde X_{B\cup\{s\},A}$ is presented by the following matrix:
\begin{equation}
\label{eq:cube-diff-mat}
d^i_{B\cup\{s\},A} =
\begin{pmatrix}
d^i_{B,A\cup\{s\}} & \widetilde\varphi_s^{i+1} \\[1.5ex]
0 & -d^{i+1}_{B,A}
\end{pmatrix}
\ ,
\end{equation}
here $\widetilde\varphi_s^i:\widetilde X_{B,A}^i\to\widetilde X_{B,A\cup\{s\}}$ is the sum of the chain maps
\[
\begin{multlined}
X_{A\cup S}\bigl[-|B\setminus S|\bigr]^i
= X_{A\cup S}^{i+|B\setminus S|} \\
\xrightarrow{(-1)^{\nu(s;S)}\varphi_s^{i+|B\setminus S|}} X_{A\cup\{s\}\cup S}^{i+|B\setminus S|}
= X_{A\cup\{s\}\cup S}\bigl[-|B\setminus S|\bigr]^i
\quad.
\end{multlined}
\]
On the other hand, we write $\widetilde\psi_s:\widetilde X_{B,A\cup\{s\}}\to\widetilde X_{B\cup\{s\},A}$ the canonical inclusion respecting the decomposition \eqref{eq:multcone-decomp}.
Then, the following result is obvious.

\begin{lemma}
\label{lem:multcone-cube}
In the above situation, the family $\{\widetilde X_{B,A}\}_{(B,A)}$ together with maps $\{\widetilde\varphi_s\}_s$ and $\{\widetilde\psi_s\}_s$ forms an $n$-cube of cofiber sequences.
\end{lemma}

Since we have the canonical identification
\[
\widetilde X_{\varnothing,A}=X_A
\]
as complexes, it follows that every cube of complexes extends to that of cofiber sequences.

\begin{definition}
\label{def:mult-mcone}
Let $X$ be an $n$-cube of complexes in $\mathcal A$.
Then, the \emph{multiple mapping cone} of $X$, written $\operatorname{MCone}(X)$, is the complex $\widetilde X_{\langle n\rangle,\varnothing}$ where $\widetilde X$ is the $n$-cube constructed above.
\end{definition}

\begin{remark}
\label{rem:multcone-orderindep}
For an $n$-cube $X$ of complexes, there may be several choices of $n$-cubes of cofiber sequences extending $X$.
It however turns out that the homotopy type of the multiple mapping cone $\operatorname{MCone}(X)$ is completely determined by $X$ itself.
\end{remark}

\begin{example}
A $1$-cube $X$ of complexes is nothing but a chain map $\varphi:X_\varnothing\to X_{\langle 1\rangle}$.
In this case, the multiple mapping cone $\operatorname{MCone}(X)$ coincides with the ordinary mapping cone $\operatorname{Cone}(\varphi)$.
\end{example}

\begin{remark}
In his paper \cite{Khovanov2000}, Khovanov defined a complex for cubes of abelian groups.
Actually, his construction essentially agrees with the multiple mapping cone (up to shift of degrees).
\end{remark}

As in the case of ordinary mapping cones, for an $n$-cube $X$ of complexes, the homotopy type of $\operatorname{MCone}(X)$ only depends on that of $X$.
Indeed, a \emph{map of $n$-cubes} $\alpha:X\to Y$ is a family of chain maps $\alpha_A:X_A\to Y_A$ indexed by subsets $A\subset\langle n\rangle$ respecting the structure maps of $n$-cubes.
In particular, we call $\alpha$ a \emph{quasi-isomorphism of $n$-cubes} if each $\alpha_A$ is a quasi-isomorphism.
The following result is proved by recursive use of \cref{prop:hinv-cofib}.

\begin{proposition}
\label{prop:hinv-mcone}
If $\alpha:X\to Y$ is a map of $n$-cubes, then it induces a chain map $\alpha_\ast:\operatorname{MCone}(X)\to\operatorname{MCone}(Y)$ which makes the square below commute:
\[
\begin{tikzcd}
X_{\langle n\rangle} \ar[r] \ar[d,"\alpha_{\langle n\rangle}"'] & \operatorname{MCone}(X) \ar[d,"\alpha_\ast"] \\
Y_{\langle n\rangle} \ar[r] & \operatorname{MCone}(Y)
\end{tikzcd}
\quad.
\]
Moreover, $\alpha_\ast$ is a quasi-isomorphism as soon as so is $\alpha$.
\end{proposition}

In \cref{sec:prop}, we are interested in vanishing of multiple mapping cone.
The following gives a simple criterion for it; the proof is easy so we leave it to the reader.

\begin{lemma}
\label{lem:cofibcube-vanish}
Let $Y=\{Y_{B,A}\}_{(B,A)}$ be an $n$-cube of cofiber sequences.
Suppose there is an element $s\in\langle n\rangle$ such that, for each $A\subset\langle n\rangle\setminus\{s\}$, the map $Y_{\varnothing,A}\to Y_{\varnothing,A\cup\{s\}}$ is a quasi-isomorphism.
Then, the complex $Y_{B,A}$ is null-homotopic provided $s\in B$.
\end{lemma}

\begin{corollary}
\label{cor:multcube-vanish}
Let $X=\{X_A\}_A$ be an $n$-cube of complexes.
Suppose there is an element $s\in\langle n\rangle$ such that the map $X_A\to X_{A\cup\{s\}}$ is a quasi-isomorphism for every subset $A\subset\langle n\rangle\setminus\{s\}$.
Then, the multiple mapping cone $\operatorname{MCone}(X)$ is null-homotopic.
\end{corollary}

\subsection{Khovanov homology of singular links}
\label{sec:KHsing:KHsing}

We apply the results in \cref{sec:KHsing:mult-cone} to build up Khovanov homology of singular links from the mapping cone of the genus-$1$ map $\widehat\Phi$.
In the application, we take $\mathcal A=(\mathbf{Vect}_{\mathbb F_2})^{\mathbb Z}$ the category of $\mathbb Z$-graded $\mathbb F_2$-vector spaces, which is abelian.

Suppose $D$ is a diagram of a singular link $L$, and put $c^\sharp(D)$ the set of double points in $D$.
For each subset $A\subset c^\sharp(D)$, put $D_A$ the diagram obtained from $D$ so that
\begin{itemize}
  \item double points which belong to $A$ are resolved to positive crossings;
  \item double points outside of $A$ are resolved to negative crossings:
\end{itemize}
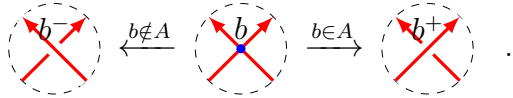
\[
\diagCrossNegUp{b^-}
\;\xleftarrow{b\notin A}\; \diagCrossSingUp{b}
\;\xrightarrow{b\in A}\; \diagCrossPosUp{b^+}
\quad.
\]
For each $b\in c^\sharp(D)\setminus A$, applying the genus-$1$ map $\widehat\Phi$ on the resolutions of $b$, we obtain a chain map
\[
\widehat\Phi_b:C(D_A)\otimes\mathbb F_2 \to C(D_{A\cup \{b\}})\otimes\mathbb F_2
\]
between Khovanov chain complexes over $\mathbb F_2$.
Note that, for distinct double points $b$ and $b'$, the maps $\widehat\Phi_b$ and $\widehat\Phi_{b'}$ commute with each other.
Hence, fixing a bijection $c^\sharp(D)\cong\{1,\dots,r\}$, we obtain an $r$-cube $\{C(D_A)\otimes\mathbb F_2\}_A$ of complexes in $(\mathbf{Vect}_{\mathbb F_2})^{\mathbb Z}$ with structure maps $\{\widehat\Phi_b\}_b$.

\begin{definition}
\label{def:Kh-sing}
Let $D$ be a diagram of a singular link.
We set
\[
\begin{gathered}
C(D;\mathbb F_2)
\coloneqq \operatorname{MCone}(\{C(D_A)\otimes\mathbb F_2\}_A,\{\widehat\Phi_b\}_b)
\\
\mathit{Kh}^{i,j}(D;\mathbb F_2)
\coloneqq H^i\left(C^{\ast,j}(D;\mathbb F_2)\right)
\end{gathered}
\]
and call them the \emph{Khovanov complex} and the \emph{Khovanov homology group} of $D$ respectively.
\end{definition}

\begin{remark}
By virtue of the results obtained in \cref{sec:invariance} and \cref{rem:multcone-orderindep}, the Khovanov homology groups $\mathit{Kh}^{i,j}(D;\mathbb F_2)$ depends neither on the choices of a diagram $D$ nor on a bijection $c^\sharp(D)\cong\{1,\dots,r\}$.
As a result, for a singular link $L$, then it is safe to write
\[
\mathit{Kh}^{i,j}(L;\mathbb F_2)
\coloneqq \mathit{Kh}^{i,j}(D;\mathbb F_2)
\]
for an arbitrary diagram $D$ of $L$.
\end{remark}

The following is obvious.

\begin{proposition}
\label{prop:KH-vs-KHsing}
In the case of non-singular links, the definition above agrees with the original one with coefficients in $\mathbb F_2$.
\end{proposition}

\section{Invariance}
\label{sec:invariance}

In this section, we prove \cref{main:KHinv} that  asserts the invariance of $Kh(\blank;\mathbb{F}_2)$ of singular links.  
We note that, according to \cite{BatainehElhamdadiHajijYoumans2018}, the moves of singular link diagrams are generated by the following three moves in addition to Reidemeister moves:
\begin{equation}
\label{eq:moveSing}
\begin{gathered}
\diagCrossSingRivOL
\;\leftrightarrow\; \diagCrossSingRivOR
\quad,\qquad
\diagCrossSingRivUL
\;\leftrightarrow\; \diagCrossSingRivUR
\quad,\\[1.5ex]
\diagRvSingL
\;\leftrightarrow\; \diagRvSingR
\quad.
\end{gathered}
\end{equation}
Thus, it will suffice to show that we have isomorphisms on Khovanov homologies along these moves.

The proof is essentially due to the following result.
For explicit description for chain maps associated to Reidemeister moves, we refer the reader to \cite[Page~337]{Viro2004}, \cite[Section~4.3]{BarNatan2005}, and \cite[Section~2]{Ito2011}.

\begin{proposition}
\label{prop:invO45}
With regard to Khovanov complex for ordinary link diagrams, the following hold:
\begin{enumerate}[label=\upshape(\arabic*)]
  \item\label{sub:invO45:O4} There are zig-zags of quasi-isomorphisms connecting the following pairs of chain maps:
\begin{gather}
\label{eq:invO45:O4a}
\begin{tikzcd}[ampersand replacement=\&]
C\left(\tikz[baseline=-.5ex]{\nodeTikzBox{diagCrossNegRivOL}; \node at (.5,0) {$b_-$};}\right) \otimes \mathbb{F}_2 \ar[r,dashed,<->,"\sim"] \ar[d,"\widehat\Phi_b"'] \& C\left(\tikz[baseline=-.5ex]{\nodeTikzBox{diagCrossNegRivOR}; \node at (-.5,0) {$b'_-$};}\right) \otimes \mathbb{F}_2 \ar[d,"\widehat\Phi_{b'}"]\\
C\left(\tikz[baseline=-.5ex]{\nodeTikzBox{diagCrossPosRivOL}; \node at (.5,0){$b_+$};}\right) \otimes \mathbb{F}_2 \ar[r,dashed,<->,"\sim"] \& C\left(\tikz[baseline=-.5ex]{\nodeTikzBox{diagCrossPosRivOR}; \node at (-.5,0) {$b'_+$};}\right) \otimes \mathbb{F}_2
\end{tikzcd}
\quad,\\
\label{eq:invO45:O4e}
\begin{tikzcd}[ampersand replacement=\&]
C\left(\tikz[baseline=-.5ex]{\nodeTikzBox{diagCrossNegRivUL}; \node at (.5,0) {$b_-$};}\right) \otimes \mathbb{F}_2 \ar[r,dashed,<->,"\sim"] \ar[d,"\widehat\Phi_b"'] \& C\left(\tikz[baseline=-.5ex]{\nodeTikzBox{diagCrossNegRivUR}; \node at (-.5,0) {$b'_-$};}\right) \otimes \mathbb{F}_2 \ar[d,"\widehat\Phi_{b'}"]\\
C\left(\tikz[baseline=-.5ex]{\nodeTikzBox{diagCrossPosRivUL}; \node at (.5,0){$b_+$};}\right) \otimes \mathbb{F}_2 \ar[r,dashed,<->,"\sim"] \& C\left(\tikz[baseline=-.5ex]{\nodeTikzBox{diagCrossPosRivUR}; \node at (-.5,0) {$b'_+$};}\right) \otimes \mathbb{F}_2
\end{tikzcd}
\quad,
\end{gather}
where $\widehat\Phi_\ast$ is the genus-$1$ map with respect to the crossing $\ast$;
  \item\label{sub:invO45:O5} The following square is commutative:
\begin{equation}
\label{eq:invR45:O5a}
\begin{tikzcd}[column sep=1em]
C\left(\tikz{\nodeTikzBox{diagRvB}; \node at (-1,0) {$b_-$};}\right) \ar[d,"\widehat\Phi_b"'] & C\left(\diagSmoothRight\right) \ar[l,"\overbar{\mathrm{RII}}"'] \ar[r,"\overbar{\mathrm{RII}}"] & C\left(\tikz{\nodeTikzBox{diagRvT}; \node at (1,0) {$b'_-$};}\right) \ar[d,"\widehat\Phi_{b'}"] \\
C\left(\tikz{\nodeTikzBox{diagRvFTw}; \node at (-1,0) {$b_+$};}\right) \ar[rr,equal] && C\left(\tikz{\nodeTikzBox{diagRvFTw}; \node at (1,0) {$b'_+$};}\right)
\end{tikzcd}
\end{equation}
\end{enumerate}
where $\overbar{\mathrm{RII}}$ is the chain map associated with Reidemeister move of type II.
\end{proposition}

The proof of the part \ref{sub:invO45:O4} in \cref{prop:invO45} relies on the invariance of Khovanov homology under Reidemeister move of type III.
We technically divide it into some lemmas; after preparing \cref{FrobeniusNotation}, we start the proofs.

\begin{notation}\label{FrobeniusNotation}
Recall that, for the TQFT $Z_A$ associated to the Frobenius algebra $A=\mathbb Z[x]/(x^2)$, each \emph{saddle point} on a cobordism causes either multiplication or comultiplication.
In order to describe them in a unified way, we use the symbol 
\[
\tikz[baseline=-.5ex]{%
  \nodeTikzBox{diagSmoothV};
  \node at (-.5,0) {$p$};
  \node at (.5,0) {$q$}; }
\;\xmapsto{m\ \text{or}\ \Delta}\;
\tikz[baseline=-.5]{%
  \nodeTikzBox{diagSmoothH};
  \node at (0,.7) {$p:q$};
  \node at (0,-.7) {$q:p$}; }
\]
to indicate the change of enhancements along a saddle operation (see \cref{sec:quick-review:enhst}).
For example, if $p=q$ are the label on the same component, then comultiplication occurs so that we have
\[
(p:q)\otimes (q:p) = p^{(1)}\otimes p^{(2)}
\]
using Sweedler notation.
\end{notation}

\begin{lemma}\label{lemma4e1}
The diagrams of chain maps below commute:
\begin{gather}
\begin{tikzcd}[ampersand replacement=\&]
C\left(\diagCrossHRivUL\right) \ar[rr,equal] \ar[d,"\delta_-"] \&\& C\left(\diagCrossHRivUR\right) \ar[d,"\delta_-"] \\
C\left(\diagCrossVRivUL\right)\left\{1\right\} \ar[r,"\mathrm{RII}"] \& C\left(\diagParallelVVV\right)\left\{1\right\} \& C\left(\diagCrossVRivUR\right)\left\{1\right\} \ar[l,"\mathrm{RII}"']
\end{tikzcd}
\quad,\\
\begin{tikzcd}[ampersand replacement=\&]
C\left(\diagCrossVRivUL\right) \ar[d,"\delta_+"] \& C\left(\diagParallelVVV\right) \ar[l,"\overline{\mathrm{RII}}"'] \ar[r,"\overline{\mathrm{RII}}"] \& C\left(\diagCrossVRivUR\right) \ar[d,"\delta_+"]
\\
C\left(\diagCrossHRivUL\right)\left\{1\right\} \ar[rr,equal] \&\& C\left(\diagCrossHRivUR\right)\left\{1\right\}
\end{tikzcd}
\quad,
\end{gather}
here $\delta_+$ and $\delta_-$ are the map which appear in the sequence \eqref{eq:delta-Phi-delta}.
\end{lemma}
\begin{proof}
One may notice that the statement is essentially the same as the proof that the isomorphism for the third Reidemeister move \cite{Ito2011} is a chain map.
\end{proof}

\begin{lemma}\label{lemma4e3}
For each pair $(i,j)$ of integers, the diagram below commutes up to sign:
\[
\begin{tikzcd}
C^{i,j}\left(\diagCrossVRivUL\right) \ar[d,"\Phi"] & C^{i,j}\left(\diagParallelVVV\right) \ar[l,"\overline{\mathrm{RII}}"'] \ar[r,"\overline{\mathrm{RII}}"] & C^{i,j}\left(\diagCrossVRivUR\right) \ar[d,"\Phi"]
\\
C^{i,j-2}\left(\diagCrossVRivUL\right) \ar[r,"\mathrm{RII}"] & C^{i,j-2}\left(\diagParallelVVV\right) & C^{i,j-2}\left(\diagCrossVRivUR\right) \ar[l,"\mathrm{RII}"']
\end{tikzcd}
\quad,
\]
here the map $\Phi$ is the one which appears in the sequence \eqref{eq:delta-Phi-delta}.
\end{lemma}
\begin{proof}   
In terms of enhanced states, the composition $\rii \circ \Phi \circ \overline{\rii}$ is given as follows:
\begin{align*}
\tikz[baseline=-.5ex]{%
  \node at (0,0) {\diagParallelVVV};
  \node[left] at (135:1) {$p$};
  \node[above] at (90:1) {$q$};
  \node[right] at (45:1) {$r$}; }
&\xmapsto{\overbar{\mathrm{RII}}}
\tikz[baseline=-.5ex]{%
  \coordinate (I1) at (135:.55);
  \coordinate (I2) at (225:.55);
  \node at (0,0) {\diagCrossVRivUL};
  \draw[blue,ultra thick] ($ (I1)+(-.05,.25) $) -- ($ (I1)+(.05,-.25) $);
  \draw[green,ultra thick] ($ (I2)+(.05,.25) $) -- ($ (I2)+(-.05,-.25) $);
  \node[left] at (135:1) {$p$};
  \node[above] at (90:1) {$q$};
  \node[right] at (45:1) {$r$}; }
+
\tikz[baseline=-.5ex]{%
  \coordinate (I1) at (135:.55);
  \coordinate (I2) at (225:.55);
  \node at (0,0) {\diagCrossVRivUL};
  \draw[green,ultra thick] ($ (I1)+(.25,.05) $) -- ($ (I1)+(-.25,-.05) $);
  \draw[blue,ultra thick] ($ (I2)+(.25,-.05) $) -- ($ (I2)+(-.25,.05) $);
  \node[above] at (120:1) {$p:q$};
  \node[below] at (240:1) {$q:p$};
  \node at (-.6,0) {$1$};
  \node[right] at (45:1) {$r$}; } \\
&\xmapsto{\Phi}
\tikz[baseline=-.5ex]{%
  \coordinate (I1) at (135:.55);
  \coordinate (I2) at (225:.55);
  \node at (0,0) {\diagCrossVRivUL};
  \draw[blue,ultra thick] ($ (I1)+(-.05,.25) $) -- ($ (I1)+(.05,-.25) $);
  \draw[green,ultra thick] ($ (I2)+(.05,.25) $) -- ($ (I2)+(-.05,-.25) $);
  \node[left] at (135:1) {$p$};
  \node[above] at (90:1) {$(q:r):(r:q)$};
  \node[right] at (45:1) {$(r:q):(q:r)$}; }
+
\tikz[baseline=-.5ex]{%
  \coordinate (I1) at (135:.55);
  \coordinate (I2) at (225:.55);
  \node at (0,0) {\diagCrossVRivUL};
  \draw[green,ultra thick] ($ (I1)+(.25,.05) $) -- ($ (I1)+(-.25,-.05) $);
  \draw[blue,ultra thick] ($ (I2)+(.25,-.05) $) -- ($ (I2)+(-.25,.05) $);
  \node[above] at (120:1) {$p:q$};
  \node[below] at (240:1) {$q:p$};
  \node at (-.8,0) {$r^{(1)}$};
  \node[right] at (45:1) {$r^{(2)}$}; } \\
&\xmapsto{\mathrm{RII}}
\tikz[baseline=-.5ex]{%
  \node at (0,0) {\diagParallelVVV};
  \node[left] at (135:1) {$p$};
  \node[above] at (90:1) {$(q:r):(r:q)$};
  \node[right] at (45:1) {$(r:q):(q:r)$}; }
+(-1)^{i+2}
\tikz[baseline=-.5ex]{%
  \node at (0,0) {\diagParallelVVV};
  \node[left] at (135:1) {$(p:q):(q:p)$};
  \node[above] at (90:1) {$(q:p):(p:q)$};
  \node[right] at (45:1) {$r$}; }
\mathrlap{\quad.}
\end{align*}
Rotating all the pictures upside-down, we also obtain the other composition.
Comparing the two, we obtain the required commutativity.
\end{proof}

\begin{proof}[Proof of \cref{prop:invO45}]
We first verify the part \ref{sub:invO45:O4}.
Note that, by virtue of the symmetry of Frobenius calculus on the second Reidemeister move, \eqref{eq:invO45:O4a} is proved in a similar manner to \eqref{eq:invO45:O4e}, so we only show the latter.
The genus-$1$ maps $\widehat\Phi$ in \eqref{eq:invO45:O4e} are induced by the sequences below respectively as asserted in \cref{cor:Phi-hat}:
\[
\begin{gathered}
C\left(\diagCrossHRivUL\right)
\xrightarrow{\delta_-} C\left(\diagCrossVRivUL\right)\left\{1\right\}
\xrightarrow{\Phi} C\left(\diagCrossVRivUL\right)\left\{3\right\}
\xrightarrow{\delta_+} C\left(\diagCrossHRivUL\right)\left\{4\right\}
\quad,\\[2ex]
C\left(\diagCrossHRivUR\right)
\xrightarrow{\delta_-} C\left(\diagCrossVRivUR\right)\left\{1\right\}
\xrightarrow{\Phi} C\left(\diagCrossVRivUR\right)\left\{3\right\}
\xrightarrow{\delta_+} C\left(\diagCrossHRivUR\right)\left\{4\right\}
\quad.
\end{gathered}
\]
On the other hand, \cref{lemma4e1} and \cref{lemma4e3} asserts that there is a zig-zag of quasi-isomorphisms connecting these sequences.
In view of \cref{prop:hinv-cofib}, it gives rise to a zig-zag of quasi-isomorphism as depicted in \eqref{eq:invO45:O4e}.

It remains to verify the part \ref{sub:invO45:O5}.
In view of the explicit formula of the chain map $\overbar{\mathrm{RII}}$, it is easy to verify that both of the two compositions are given, in terms of enhanced states, as follows:
\[
\tikz[baseline=-.5ex]{%
  \node at (0,0) {\diagSmoothRight};
  \node at (0,.7) {$p$};
  \node at (0,-.7) {$q$};
}
\;\mapsto\;
\tikz[baseline=-.5ex]{%
  \node at (0,0) {\diagRvFTw};
  \draw[green, ultra thick] (-.8,0) -- (-.2,0);
  \draw[green, ultra thick] (.2,0) -- (.8,0);
  \node at (0,1.2) {$(p:q):(q:p)$};
  \node at (0,-1.2) {$(q:p):(p:q)$};
}
\quad.
\]
\end{proof}

We are now ready to prove \cref{main:KHinv}.

\begin{proof}[Proof of \cref{main:KHinv}]
As mentioned in the beginning of the section, it suffices to show invariance under the moves \eqref{eq:moveSing}.
Let $D$ and $D'$ be two singular link diagrams with exactly $r$ double points such that they are related by one of the moves \eqref{eq:moveSing} or Reidemeister moves; hence we have a canonical identification $c^\#(D)=c^\#(D')$.
Put $\{C(D_A)\otimes\mathbb F_2\}_{A\subset c^\#(D)}$ and $\{C(D'_A)\otimes\mathbb F_2\}_{A\subset c^\#(D')}$ to be the $r$-cube constructed in \cref{sec:KHsing:KHsing}.
Since Reidemeister moves and the moves \eqref{eq:moveSing} are realized as chain maps induced by local operations of TQFT, using \cref{prop:invO45} and explicit descriptions in \cref{lemma4e1} and \cref{lemma4e3}, one sees that there is a zig-zag of quasi-isomorphisms of $r$-cubes depicted as below:
\[
\begin{tikzcd}
\{C(D_A)\otimes\mathbb F_2\}_{A\subset c^\#(D)} \ar[r,<->,dashed,"\sim"] & \{C(D'_A)\otimes\mathbb F_2\}_{A\subset c^\#(D')}
\end{tikzcd}
\quad.
\]
By \cref{prop:hinv-mcone}, this induces an isomorphism on the homologies of multiple mapping cones; namely,
\[
\mathit{Kh}^{\ast,\star}(D;\mathbb F_2)
\cong \mathit{Kh}^{\ast,\star}(D';\mathbb F_2)
\quad.
\]
This completes the proof.
\end{proof}

\section{Properties of extended Khovanov homology}
\label{sec:prop}

In this last section, we investigate properties of our extension of Khovanov homology.
In particular, the remaining results \cref{main:catVas} and \cref{main:KH-FIrel} are proved.

\subsection{Categorified Vassiliev skein relation}
\label{sec:cat-Vassiliev}

We first give a proof of \cref{main:catVas} that asserts Khovanov homology satisfies a categorified version of Vassiliev skein relation.
The proof relies on the general results on multiple mapping cones.

\begin{proof}[Proof of \cref{main:catVas}]
Let $D$ be a diagram of $L$, so we obtain an $r$-cube $\{C(D_A)\otimes\mathbb F_2\}_A$ of complexes in the way described above.
Let $\widetilde X=\{\widetilde X_{B,A}\}_{(B,A)}$ be an $r$-cube of cofiber sequences extending $\{C(D_A)\otimes\mathbb F_2\}_A$ obtained in the way described before \cref{lem:multcone-cube}.
For the double point $b\in c^\sharp(D)$, we write $D_+$ and $D_-$ the diagrams obtained from $D$ by resolving $b$ to positive and negative crossings respectively.
Then, unwinding the construction of $\widetilde X$ described in \cref{sec:KHsing:mult-cone}, one sees that we have a cofiber sequence
\[
C(D_-;\mathbb F_2)
\xrightarrow{\widehat\Phi} C(D_+;\mathbb F_2)
\to C(D;\mathbb F_2)
\ .
\]
Taking the homology groups, we obtain the required long exact sequence.
\end{proof}

As a consequence of \cref{main:catVas}, the unnormalized Jones polynomial is recovered from our extension of Khovanov homology.
We write $\widetilde V$ the unnormalized Jones polynomial and $\widetilde V^{(r)}$ its $r$-th derivative.

\begin{corollary}
\label{cor:Khsing-Jones}
For a singular link $L$ with exactly $r$ double points, we have
\[
\left[\mathit{Kh}(L;\mathbb F_2)\right]_q
= \widetilde V^{(r)}(L)
\quad.
\]
\end{corollary}
\begin{proof}
We prove the statement by induction on the number $r$ of double points.
If $r=0$, i.e. $L$ is an ordinary link, then the statement is exactly \cref{theo:Khovanov}.
Assume $r>0$ and the statement is true for links with less than $r$ double points and suppose $L$ is a link with exactly $r$ double points.
Take a double point $b$ of $L$, and write $L_+$ and $L_-$ (possibly singular) links obtained from $L$ by resolving $b$ into positive and negative crossings respectively.
Then \cref{main:catVas} implies that there is a long exact sequence of the form
\[
\cdots
\to\mathit{Kh}^{i-1,\star}(L;\mathbb F_2)
\to \mathit{Kh}^{i,\star}(L_-;\mathbb F_2)
\xrightarrow{\widehat\Phi} \mathit{Kh}^{i,j}(L_+;\mathbb F_2)
\to\mathit{Kh}^{i,\star}(L;\mathbb F_2)
\to\cdots
\quad.
\]
Taking the Euler characteristics, we obtain the equations
\begin{equation}
\label{eq:prf:Khsing-Jones:Vskein}
\left[\mathit{Kh}^{i,j}(L;\mathbb F_2)\right]_q
=\left[\mathit{Kh}^{i,j}(L_+;\mathbb F_2)\right]_q
-\left[\mathit{Kh}^{i,j}(L_-;\mathbb F_2)\right]_q
\quad.
\end{equation}
By induction hypothesis, \eqref{eq:prf:Khsing-Jones:Vskein} is exactly Vassiliev skein relation for the link $L$ with respect to the double point $b$.
Thus, this completes the induction and hence the statement for all links.
\end{proof}

\subsection{FI relation}
\label{sec:KHsing:FI}

We next show that Khovanov homology satisfies the \emph{FI relation}.
For a link invariant $v$ extended to singular links, Vassiliev skein relation \eqref{eq:V-skein} yields the equation
\[
v\left(\diagFiSing\right)
= v\left(\diagFiPos\right)
- v\left(\diagFiNeg\right)
= 0
\quad.
\]
We prove \cref{main:KH-FIrel} by categorifying this equation.

\begin{proof}[Proof of \cref{main:KH-FIrel}]
Let $D$ be a diagram of $L$, and suppose $b\in c^\sharp(D)$ is a double point of the form \eqref{eq:KH-FIrel:FIsing}.
By the construction of the complex $C(D;\mathbb F_2)$, and by virtue of \cref{cor:multcube-vanish}, it suffices to verify the statement only in the case $c^\sharp(D)=\{b\}$.

Put $D_+$ and $D_-$ the diagram obtained from $D$ by resolving $b$ to positive and negative crossings respectively.
We denote by $\widehat\Phi_b:C(D_-;\mathbb F_2)\to C(D_+;\mathbb F_2)$ the genus-$1$ map on $b$ and consider the following diagram:
\begin{equation}
\label{eq:prf:KH-FIrel:tri}
\begin{tikzcd}[column sep=0pt]
& C\left(\diagRiNil\;;\mathbb F_2\right) \ar[dl,"\mathrm{RI}^-"'] \ar[dr,"\mathrm{RI}^+"] & \\
C\left(\diagFiNeg\;;\mathbb F_2\right) \ar[rr,"\widehat\Phi_b"] && C\left(\diagFiPos\;;\mathbb F_2\right)
\end{tikzcd}
\quad,
\end{equation}
where $\mathrm{RI}^+$ and $\mathrm{RI}^-$ are the maps corresponding to the left-hand and right-hand Reidemeister moves of type I respectively.
According to \cite{Viro2004}, they are explicitly described in terms of enhanced states (see \cref{sec:quick-review:enhst}) as follows:
\[
\begin{gathered}
\mathrm{RI}^-\left(
\tikz[baseline=-.5ex]{%
  \nodeTikzBox{diagRiNil};
  \node at (-.4,0) {$p$}; }
\right)
\;=\;
\tikz[baseline=-.5ex]{%
  \nodeTikzBox{diagRiNeg};
  \draw[blue,ultra thick] (0,.4) -- (0,-.4);
  \node at (-.6,0) {$p$};
  \node at (.4,0) {$1$}; }
\quad,
\\
\mathrm{RI}^+\left(
\tikz[baseline=-.5ex]{%
  \nodeTikzBox{diagRiNil};
  \node at (-.4,0) {$p$}; }
\right)
\;=\;
\tikz[baseline=-.5ex]{%
  \nodeTikzBox{diagRiPos};
  \draw[green,ultra thick] (0,.4) -- (0,-.4);
  \node at (-.6,0) {$p$};
  \node at (.4,0) {$x$}; }
+
\tikz[baseline=-.5ex]{%
  \nodeTikzBox{diagRiPos};
  \draw[green,ultra thick] (0,.4) -- (0,-.4);
  \node at (-.6,0) {$px$};
  \node at (.4,0) {$1$}; }
\quad.
\end{gathered}
\]
It then turns out that the diagram \eqref{eq:prf:KH-FIrel:tri} is commutative.
In particular, since both $\mathrm{RI}^-$ and $\mathrm{RI}^+$ are quasi-isomorphisms, $\widehat\Phi_b$ is also a quasi-isomorphism.
On the other hand, we have a cofiber sequence
\[
C(D_-;\mathbb F_2)
\xrightarrow{\widehat\Phi_b} C(D_+;\mathbb F_2)
\to C(D;\mathbb F_2)
\ .
\]
It follows that $C(D;\mathbb F_2)$ is null-homotopic, and we obtain the result.
\end{proof}

\section*{Acknowledgements}
The authors would like to thank Professor Toshitake Kohno for his comments.

\end{document}